\documentclass[reqno, 11pt]{amsart}

\usepackage[letterpaper,hmargin=1in,vmargin=1.2in]{geometry}
\usepackage{amsmath, amssymb, amsthm, verbatim, url}
\usepackage{graphicx}
\usepackage{enumerate}
\usepackage{amsmath,amssymb,amsthm,mathrsfs,graphicx,url}
\usepackage[usenames,dvipsnames]{color}
\usepackage[colorlinks=true,linkcolor=Red,citecolor=Green]{hyperref}

\newcommand{\bs}{\backslash}

\newcommand \R {\mathbb{R}}

\DeclareMathOperator \re {Re}
\DeclareMathOperator \im {Im}

\DeclareMathOperator \supp {supp}

\newtheorem{prop}{Proposition}
\newtheorem{lem}[prop]{Lemma}

\newtheorem*{thm}{Theorem}
\numberwithin{equation}{section}
\numberwithin{prop}{section}
\numberwithin{figure}{section}

\author{Kiril Datchev}
\email{kdatchev@purdue.edu}
\address{Department of Mathematics, Purdue University,
West Lafayette, IN 47907, USA}

\author{Maarten V. de Hoop}
\email{mdehoop@purdue.edu}
\address{Department of Mathematics, Purdue University,
West Lafayette, IN 47907, USA}

\title[Iterative reconstruction  for the wave equation with bounded
frequency data]{Iterative reconstruction of the wavespeed for the wave equation with bounded
frequency boundary data}

\begin{document}

\begin{abstract}

We study the inverse boundary value problem for the wave equation
using the single-layer potential operator as the data. We assume that
the data have frequency content in a bounded interval. We prove how to
choose classes of nonsmooth coefficient functions so that optimization
formulations of inverse wave problems satisfy the prerequisites for
application of steepest descent and Newton-type iterative methods.
\end{abstract}

\maketitle \setcounter{page}{1}

\section{Introduction}

In this paper, we study the inverse boundary value problem for the
wave equation using the single-layer potential operator as the
data. We assume that the data have frequency content in a bounded
interval. The mentioned inverse boundary value problem arises, for example, in
reflection seismology \cite{Bao2005a, Symes2009, Bao2010}. 

We show how to choose classes of nonsmooth coefficient functions so
that optimization formulations of inverse wave problems satisfy the
prerequisites for application of steepest descent and Newton-type
iterative methods. Indeed, we establish the existence of a misfit
functional derived from the Hilbert-Schmidt norm and its gradient. The
proof is based on resolvent estimates for the corresponding Helmholtz
equation, exploiting the fact that the frequencies are contained in a bounded
interval.

Via conditional Lipschitz stability estimates for the time-harmonic
inverse boundary value problem, which we established in earlier work
\cite{BdHQ}, we can then guarantee convergence of the iteration
if it is initiated within a certain distance of the (unique) solution
of the inverse boundary value problem. Indeed, such a convergence of a
nonlinear projected steepest descent iteration was obtained in
\cite{dqs}.

In our scheme we can allow approximate localization of the data in
selected time windows, with size inversely proportional to the maximum allowed frequency. This is of importance to applications
in the context of reducing the complexity of field data and thus of
the underlying coefficient functions. Note that no information is lost by cutting out a short time window, since our source functions (and solutions), being compactly supported in frequency, are analytic with respect to time.

Uniqueness of the mentioned inverse boundary value problem for the
Helmholtz equation, that is, using single-frequency data, was
established by Sylvester and Uhlmann \cite{su} assuming that
the wavespeed is a bounded measurable function. This inverse problem
has also been extensively studied from an optimization point of
view. We mention, in particular, the work of \cite{Hadj-Ali2008}.

This paper can be viewed as a counterpart of the work by Blazek, Stolk
\& Symes \cite{bss} in the sense that we consider bounded frequency
data. That is, we cannot allow arbitrarily high frequencies in the
data. Without this restriction they observed that the adjoint equation did not admit
solutions: this problem does not appear in our formulation through the
use of resolvent estimates.

\subsection*{Multi-frequency data}

The multi-frequency data are obtained from solutions to the
corresponding boundary value problem for the wave equation by applying
a Fourier transform (see \cite{Lasiecka1983} for regularity of hyperbolic equations in such settings). Let $\Omega$ be a bounded smooth domain in
$\mathbb{R}^3$ and $c \in L^\infty(\mathbb R^3)$ be a strictly positive bounded
measurable function, constant outside of $\Omega$.  We consider the inhomogeneous problem for the
wave equation
\begin{equation}\label{e:wave}
\partial_t^2 u - c^2 \Delta u = f dS , 
\end{equation}
where $f \in L^2( \mathbb R; H^{-1/2}(\partial \Omega))$ is compactly supported in frequency, and $dS$ is the surface measure on $\partial \Omega$ inherited from $\mathbb R^3$. 

Multi-frequency data have been exploited in so-called frequency
progression in iterative schemes for the purpose of regularization
\cite{Bunks1995, Sirgue2004, Bao2007, Chen2009}. Frequency progression
can be realized in our approach by gradually enlarging the frequency
interval of the data; see also \cite{Simons}.

\subsection*{Reflection seismology and optimization}

Iterative methods for the inverse boundary value problem for the wave
equation, in reflection seismology, have been collectively referred to
as full waveform inversion (FWI). (The term `full waveform inversion'
was presumably introduced by Pan, Phinney and Odom in \cite{Pan1988}
with reference to the use of full seismograms information.) Lailly
\cite{Lailly1983} and Tarantola \cite{Tarantola1984, Tarantola1987}
introduced the formulation of the seismic inverse problem as a local
optimization problem with a least-squares ($L^2$) minimization of a
misfit functional. We also mention the original work of Bamberger,
Chavent \& Lailly \cite{Bamberger1977, Bamberger1979} in the
one-dimensional case. Since then, a range of alternative misfit
functionals have been considered; we mention, here, the criterion
derived from the instantaneous phase used by Bozdag, Trampert and
Tromp \cite{Bozdag2011}. The time-harmonic formulation was initially
promoted by Pratt and collaborators in \cite{Pratt1990, Pratt1991}.
Later the use of complex frequencies was studied in \cite{Shin2009,
  Ha2010}. In FWI one commonly applies a ``nonlinear'' conjugate
gradient method, a Gauss-Newton method, or a quasi-Newton method
(L-BFGS; for a review, see Brossier \cite{Brossier2009}).

\subsection*{Single-layer potential operator as a Hilbert-Schmidt
             operator}
The
introduction of the single-layer potential operator is motivated by
what seismologists refer to as the process of source
``blending''. This becomes clear upon introducing a Hilbert-Schmidt
norm for this operator, which we justify in the development of a
misfit criterion: Basis functions of the underlying Hilbert space are
viewed as blended sources. The use of ``simultaneous'' sources in
linearized inverse scattering was studied by Dai and Schuster
\cite{Dai2009}, and in full waveform inversion, for example, by Vigh
and Starr \cite{Vigh2008} (synthesizing source plane waves), Krebs
\textit{et al.}  \cite{Krebs2009} (random source encoding) and Gao,
Atle and Williamson \cite{Gao2010} (deterministic source encoding).

Berkhout, Blacquire and Verschuur \cite{Berkhout2008-1,Berkhout2008-2}
considered simple time delays for the blending process, allowing the
use of conventional sources in acquisition. The process of source
blending has appeared in various acquisition (and imaging)
strategies. Perhaps the most basic form involves synthesizing source
plane waves from point source data in plane-wave migration
\cite{Whitmore1995}. So-called controlled illumination
\cite{Rietveld1994} can also be viewed as a particular blending
strategy. In blended acquisition, typically, time-overlapping point
source experiments, are generated in the field by
using incoherent source arrays; for simultaneous source firing, see
Beasley, Chambers and Jiang \cite{Beasley1998} and for near
simultaneous source firing, see Stefani, Hampson and Herkenhoff
\cite{Stefani2007}. The use of simultaneous random sources have been
proposed, further, by \cite{Neelamani2010} and others.

\subsection*{Conditional Lipschitz stability estimates}

It is well known that the logarithmic character of stability of the
inverse boundary value problem for the Helmholtz equation
\cite{Alessandrini1988, Novikov2011} cannot be avoided. In fact, in
\cite{Mandache2001} Mandache proved that despite of regularity or
a-priori assumptions of any order on the unknown wavespeed,
logarithmic stability is optimal. However, conditional Lipschitz
stability estimates can be obtained: for example, accounting for discontinuities,
such an estimate holds if the unknown wavespeed is a finite linear
combination of piecewise constant functions with an underlying known
domain partitioning \cite{BdHQ}. It was obtained following an
approach introduced by Alessandrini and Vessella
\cite{Alessandrini2005} and further developed by Beretta and Francini
\cite{Beretta2011} for Electrical Impedance Tomography (EIT).

The relationship between the single-layer potential operator and the
Dirichlet-to-Neumann map can be found in Nachman
\cite{Nachman1988-1}. Using this relationship, it follows that
conditional Lipschitz stability using the Dirichlet-to-Neumann map as
the data implies conditional Lipschitz stability using the
single-layer potential operator as the data. Note that this stability result is the only place in our proof that we need the Dirichlet-to-Neumann map.

\subsection*{Resolvent estimates}

We control the forward operator via resolvent estimates for the Helmholtz equation. In our low-regularity setting it is well known that the resolvent norm may go to infinity exponentially in frequency as energy goes to infinity: a famous example is the square well potential which goes back to Gamow (see e.g. \cite[Theorem 2.25]{dz}). It is known in very general smooth settings that this growth is the worst that may occur: see results of Burq \cite{burq, bu}, Cardoso and Vodev \cite{cv}, and Rodnianski and Tao \cite{rt}. In \cite{d} this is proved in a lower regularity setting, and in \S\ref{s:res} below we give a generalization of this result to certain piecewise constant wavespeeds. We also give an improved estimate by cutting off away from a sufficiently large compact set: this kind of improvement has been observed before in \cite{bu, cv, rt}.

\subsection*{Main result}

\begin{figure}
\includegraphics[width=13cm]{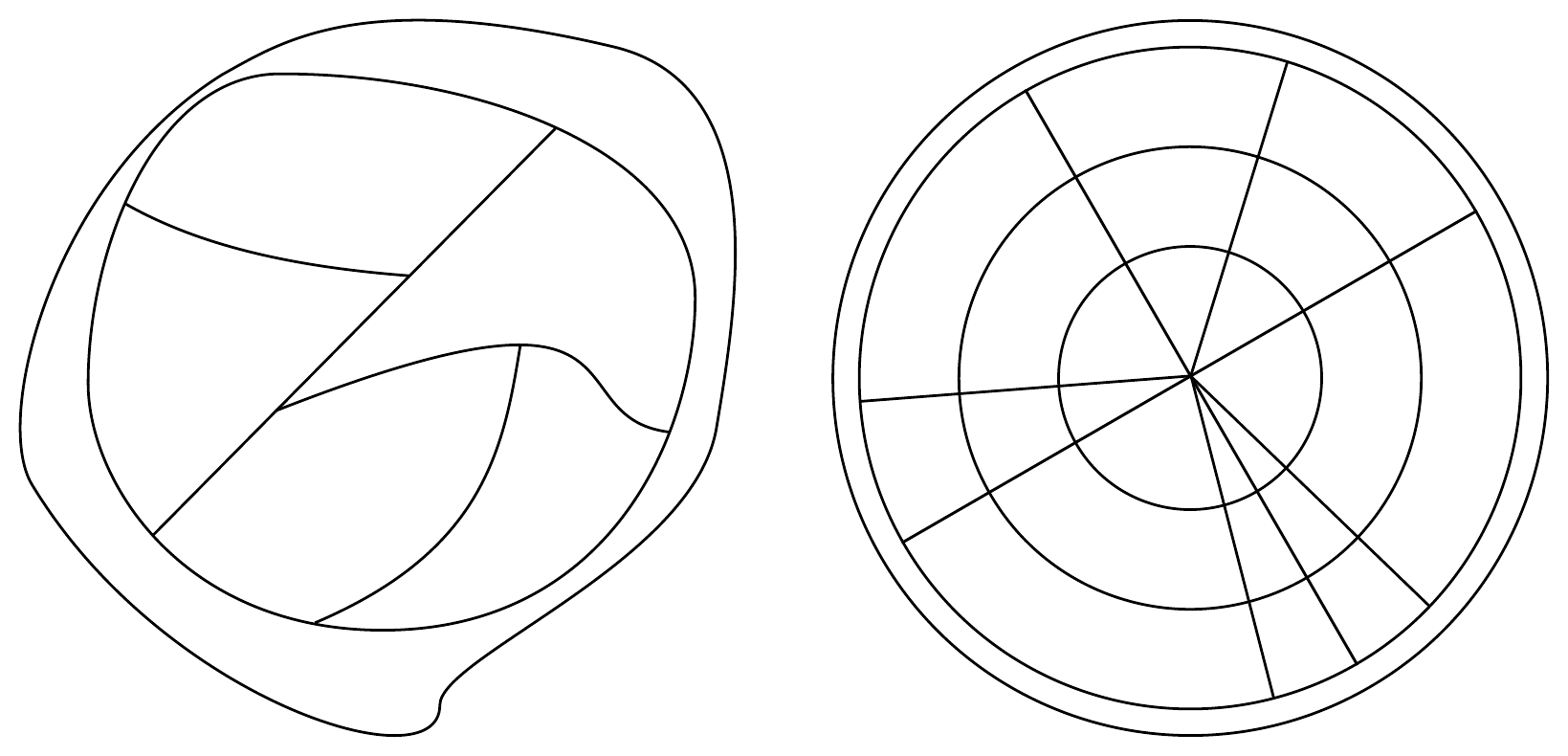}
\caption{The Theorem applies to wavespeeds which are piecewise constant with respect to subdomains as shown. The example on the left has $N=5$ and the example on the right has $N=24$. Part (1) of the Theorem (which says that the iteration converges exponentially fast) applies to both examples, but part (2) (which gives bounds for large frequencies) applies only to the example on the right.}
\end{figure}

Let $\Omega$ be a bounded domain in $\mathbb R^3$ with smooth boundary, and let $\{D_0, \dots D_N\}$ be a finite collection pairwise disjoint open subsets of $\Omega$ with Lipschitz boundaries, satisfying $\cup_{j=0}^N \overline{D_j} = \overline \Omega$ and  $\overline D_j \cap \partial \Omega \ne \varnothing \Rightarrow j=0$.

Given positive real numbers $b_1, \dots b_N$, define a wavespeed $c$ given by $b_j$ on  $D_j$ and $1$ otherwise, and consider the following forward solution to \eqref{e:wave}:
\[
u(t):= U_cf(t):= \frac 1 {2\pi} \int e^{-it\lambda}(-c^2\Delta - (\lambda+i0)^2)^{-1}\hat f (\lambda) dS d\lambda, \qquad \hat f (\lambda): = \int e^{-it\lambda}f(t)dt,
\]
where $\hat f \in L^2(\mathbb R; H^{-1/2}(\partial \Omega))$ vanishes for $\lambda \not\in[-\lambda_0,\lambda_0]$ (see \eqref{e:udef}). Throughout the paper we take $\lambda_0\ge 1$.

Letting $\tau_{\partial \Omega}$ denote restriction to the boundary $\partial \Omega$, we take as our data knowledge of the operator
\[
F\colon X_0 \to Y, \qquad F(b_1,\dots b_N)= \tau_{\partial \Omega}(U_c - U_1),
\]
where $X_0: = (0,\infty)^N \subset X: = \mathbb R^N$ and $Y$ is a tensor product of $L^2$ functions on a time window with a Hilbert space of Hilbert--Schmidt operators. More specifically
\[
\|c\|_{X}^2 = \sum_{j=1}^N |b_j|^2, \qquad  \|F(c)\|_Y^2 = \sum_{j=1}^\infty \int_I \|\tau_{\partial \Omega} (U_{c} - U_{1})\psi_j\|^2_{H^{1/2}(\partial \Omega)} dt.
\]
Here $U_1$ is the free wave evolution (when $c \equiv 1$),  $\{\psi_j\}_{j=1}^\infty$ is any orthonormal basis of the space of sources $f$,  and $I$ is any interval of length $\pi/\lambda_0$ (see \S\ref{s:land}). Let $c^\dagger \in X_0$ be the ``true wavespeed''  which we wish to recover, and define Landweber iterates by
\[
c_{m+1} = c_m - \mu DF(c_m)^*(F(c_m) - F(c^\dagger)),
\]
for some $c_0 \in X_0$ and $\mu>0$ sufficiently small, where $DF(c_m)^*$ is the gradient of $F$ at $c_m$  (see \eqref{e:landit}). 

\begin{thm} Locally, the Landweber iteration converges exponentially fast. More specifically:
 \begin{enumerate} 
 \item If $\|c_0 - c^\dagger\|_X \le C_0$ and $C_0, \ \mu>0$ are sufficiently small, then there is $C_1$ such that $\|c_m - c^\dagger\|_X \le C_0 e^{-m/C_1}$. 
 \item If the boundaries of the subdomains $D_j$ respect polar coordinates, in the sense that the radial derivative of any wavespeed under consideration is supported on a union of spheres, then $C_0^{-1}$, $\mu^{-1}$, and  $C_1$, all grow at most exponentially in $\lambda_0$.
 \end{enumerate}
\end{thm}

Here $\psi_j$ can be viewed as the ``simultaneous'' sources in ``blended''
data acquisition; we should perhaps emphasize that there is no need to
extract the single-layer potential operator, that is to say no need for ``deblending''.

\subsection*{Outline of the paper} In \S\ref{s:dtnslp} we review some known results about the Helmholtz equation with a bounded measurable potential function, including the relationship between the Dirichlet-to-Neumann map and the single-layer potential operator. In \S\ref{s:res} we give high-energy resolvent estimates for certain classes of wavespeeds. In \S\ref{s:forward} we study the resolvent for a bounded measurable wavespeed, we introduce a forward solution operator for the wave equation with ``bandlimited'' data, and we compute its Fr\'echet derivative. In \S\ref{s:land} we give an abstract setting for a Landweber iteration, which reconstructs the wavespeed from the forward operator. We build in a time localization, and control the forward operator in a Hilbert--Schmidt sense. We compute the misfit functional and gradient in terms of the corresponding weighted $L^2$ inner product in time and a Hilbert--Schmidt inner product in space. In \S\ref{s:conv} we apply the results of the previous sections to prove the Theorem.

For most of the paper we work in greater generality than the setting described above in the statement of the Theorem: see the beginning of each section for the assumptions used in that section.

The authors are grateful to Maciej Zworski for helpful discussions about resolvent estimates.

\section{Modelling time-harmonic data: Dirichlet-to-Neumann map versus single-layer
         potential operator}\label{s:dtnslp}

In this section we review the relationship between the Dirichlet-to-Neumann map and the single-layer potential operator. In this paper we are principally concerned with the latter, and use the former only for stability estimates. It is only because we use the Dirichlet-to-Neumann map at all that, at a few places in our proof, we must assume we are away from the Dirichlet spectrum of certain operators on $\Omega$.  We expect stability estimates to hold for the single-layer potential operator itself without reference to the Dirichlet-to-Neumann map, and with such estimates in hand one would be able to remove all reference to Dirichlet spectra from the proof.

Here, we consider time-harmonic waves, described by solutions, $u$
say, of the Helmholtz equation on a bounded open domain $\Omega
\subset \R ^n$.  We write
\[
   \tilde{q}(\tilde x) = -\lambda^2 c^{-2}(\tilde x) ,
\]
where $\tilde x$ is a coordinate on $\mathbb R^n$ (we will use the notation $x$ without tilde for an element of a Hilbert space later in \S\ref{s:land}),
and keep $\lambda \in \R$ fixed; we assume that $c \in L^\infty
(\Omega)$ and that $c$ is bounded below by a positive constant. Note that in the rest of the paper we use the resolvent $R_c(\lambda)=(-c(\tilde x)^2 \Delta - \lambda^2)$, whose integral kernel differs from the Green's function of this section by a factor of $c(\tilde x)^2$.   We have the general formulation
\begin{equation}\label{Helmholtz-Lp}
\left\{
\begin{array}{rcl}
   (-\Delta + \tilde{q}(\tilde x))u &=& 0,\quad \tilde x \in \Omega ,
\\
   u &=& g ,\quad \tilde x \in \partial \Omega .
\end{array}
\right.
\end{equation}
Here, $g = g(\tilde x,\lambda)$ is a boundary source.

The data generated by the boundary sources, $g$, represent the
Dirichlet-to-Neumann map $\Lambda_q$ such that
\begin{equation*}
   \Lambda_{q} :\ H^{1/2}(\partial \Omega)
             \rightarrow H^{-1/2}(\partial \Omega),\quad
   g \to \left.
   \dfrac{\partial u}{\partial \nu} \right|_{\partial \Omega} ,
\end{equation*}
where $\nu$ represents the outward unit normal to $\partial \Omega$.
We assume that the boundary $\partial\Omega$ is in $C^{(1,1)}$ and
$\Omega' = \R^n \bs \overline{\Omega}$ is connected. We also assume
that $0$ is not a Dirichlet eigenvalue for $ -\Delta + \tilde{q}$ in
$\Omega$.

Seismic reflection data are generated by point sources on
$\partial\Omega$ and observed at points on $\partial\Omega$. In
preparation of a description of the data in terms of fundamental
solutions in $\R^n$, we extend $\tilde{q}(\tilde x)$ to a function with value
\begin{equation}
   -k^2 = -\lambda^2 c_0^{-2}
\end{equation}
in $\Omega'$. Let $G^+_k(\tilde x,\tilde y)$ be the outgoing Green's function for
the Helmholtz equation with constant coefficient, $c_0^{-2}$, in
$\R^n$, which is given by
\begin{eqnarray}
   G^+_k(\tilde x,\tilde y) & =& \frac{1}{(2 \pi)^n}
        \int \frac{e^{i (\tilde x-\tilde y) \xi}}{\xi^2 - k^2 - i 0} d \xi
\\
   &=& \frac{i}{4}
      \left(\frac{|k|}{2\pi |\tilde x-\tilde y|}\right)^{(n-2)/2}
               H^{(1)}_{(n-2)/2} (|k||\tilde x-\tilde y|) .
\nonumber
\end{eqnarray}
We set
\[
   q(\tilde x) = \tilde{q}(\tilde x) + k^2 ,
\]
which is compactly supported. We assume that $k^2$ is not a Dirichlet
eigenvalue of $-\Delta + q$ or of $-\Delta$ in $\Omega$. We let
$\mathcal{G}_{q,k}(\tilde x,\tilde y)$ be the solution of
\begin{equation}
   (-\Delta_{\tilde x} + q - k^2) \,  \mathcal{G}_{q,k}(\tilde x, \tilde y)
           = \delta(\tilde x-\tilde y) ,\quad \tilde x,\tilde y \in \R^n ,
\end{equation}
satisfying the Sommerfeld radiation condition as
$|\tilde x|\rightarrow\infty$. Restricting $\tilde x$ and $\tilde y$ to $\partial\Omega$
then yields the seismic reflection data:
\[
   \mathcal{A} = \{ \mathcal{G}_{q^{\dagger},k}( \tilde x,\tilde y) \mid\
           \tilde x, \tilde y \in \partial\Omega ,\ \tilde x \ne \tilde y \} ,
\]
if $q^{\dagger}(\tilde x)$ signifies the ``true'' model.

In the constant ``reference'' model with wave speed $c_0$, we
introduce the operator,
\[
   S^+_k :\ H^{1/2}(\partial\Omega) \to
                         H^{3/2}(\partial\Omega) ,
\]
by
\begin{equation}
   S^+_k w(\tilde x) = \int_{\partial\Omega} G^+_k(\tilde x,\tilde y) \, w(y)
                  \, d S(y) ,\quad \tilde x \in \partial\Omega ,
\end{equation}
which is bounded. Here, $d S$ is the natural area element on
$\partial\Omega$. In a general heterogeneous model, we introduce
\[
   \mathcal{S}_{q,k} :\ H^{1/2}(\partial\Omega) \to
                         H^{3/2}(\partial\Omega) ,
\]
with
\begin{equation}
   \mathcal{S}_{q,k} w(\tilde x) = \int_{\partial\Omega}
                 \mathcal{G}_{q,k}(\tilde x, \tilde y) \, w(\tilde y)
                 \, d S (\tilde y) ,\quad \tilde x \in \partial\Omega ,
\end{equation}
which is bounded also \cite[Theorem~1.6]{Nachman1988-1}. Data
generated by or synthesized with simultaneous sources are then
represented by
\[
   \mathcal{B} = \{ \mathcal{S}_{q^{\dagger},k} w \mid\
                       w \in H^{1/2}(\partial\Omega) \} .
\]

The single-layer potential operator, $S_{q^{\dagger},k}$, or data
$\mathcal{B}$, are equivalent to the Dirichlet-to-Neumann map,
$\Lambda_{q^{\dagger} - k^2}$, in the sense that they contain the same
information about $q^{\dagger}$. Indeed, from $S_k^+$, we can build
the relation between $\mathcal{S}_{q,k}$ and the Dirichlet-to-Neumann
map. We have
\begin{equation} \label{iden}
   \Lambda_{q-k^2} =  \Lambda_{-k^2}
            + \mathcal{S}_{q,k}^{-1} - (S_k^+)^{-1} ,
\end{equation}
or
\begin{equation}
  \mathcal{S}_{q,k} - S_k^+ 
      = -S_k^+ \,
            (\Lambda_{q-k^2} - \Lambda_{-k^2}) \, \mathcal{S}_{q,k} .
\end{equation}
This identity is defined on $H^{1/2}(\partial\Omega)$, and can be
derived from the resolvent equation,
\begin{equation}
   \mathcal{G}_{q,k}(\tilde x,\tilde y) = G^+_k (\tilde x-\tilde y)
   - \int_{\Omega} G^+_k (\tilde x-\tilde z) q(\tilde z) \mathcal{G}_{q,k} (\tilde z,\tilde y) \, d \tilde z .
\end{equation}
For $w \in H^{1/2}(\partial\Omega)$, we then find that
\begin{equation}
   \mathcal{S}_{q,k} w(\tilde x)-S^+_k w(\tilde x) =
     - \int_\Omega G^+_k(\tilde x-\tilde z) q(\tilde z) (\mathcal{S}_{q,k} w)(\tilde z) \, d \tilde z .
\end{equation}
From (\ref{iden}) we straightforwardedly obtain
\begin{equation}
   \Lambda_{q-k^2} - \Lambda_{q^{\dagger}-k^2}
      = \mathcal{S}_{q,k}^{-1} - \mathcal{S}_{q^{\dagger},k}^{-1} ,
\end{equation}
or
\begin{equation*}
   \Lambda_{q-k^2} - \Lambda_{q^{\dagger}-k^2}
   = -\mathcal{S}_{q^{\dagger},k}^{-1}
      (\mathcal{S}_{q,k} - \mathcal{S}_{q^{\dagger},k})
       \mathcal{S}_{q,k}^{-1} .
\end{equation*}

\subsection*{Conditional Lipschitz stability}

The convergence rate and convergence radius of our iterative scheme
are based on a conditional Lipschitz-type stability estimate for the
inverse problem:
\begin{equation}
   \| c^{-2} - {c^{\dagger}}^{-2} \|
   \leq C_S\, 
     \| \mathcal{S}_{\lambda^2 {c}^{-2},k} -\mathcal{S}_{\lambda^2 {c^\dagger}^{-2},k}  \|,
\label{eq:lip}
\end{equation}
where $c^{\dagger}$ is the true wavespeed. In case the wave speed is
piecewise constant \cite{BdHQ}, the above holds using the
Dirichlet-to-Neumann map as the data. However, if the inverse boundary
value problem with the Dirichlet-to-Neumann map as the data is
Lipschitz stable, then the inverse problem with the single-layer
potential operator as the data is Lipschitz stable.

Indeed, assume that $k^2$ is not a Dirichlet eigenvalue of $-\Delta +
q$ or of $-\Delta$ in $\Omega$. Then the inverse, $(S_k^+)^{-1}$, of
operator $S_k^+$ exists and is bounded, $H^{3/2}(\partial\Omega) \to
H^{1/2}(\partial\Omega)$. Moreover, the inverse,
$\mathcal{S}_{q,k}^{-1}$, of operator $\mathcal{S}_{q,k}$ exists and
is bounded, $H^{3/2}(\partial\Omega) \to H^{1/2}(\partial\Omega)$.
For a proof of this Proposition, see
\cite[Section~6]{Nachman1988-1}. Essentially, it follows that
\[
   \mathcal{S}_{q,k} = S_k^+ [I + (S_k^+)^{-1} (\mathcal{S}_{q,k} -
     S_k^+)]
\]
is invertible by showing that $-1$ cannot be an eigenvalue of
$(S_k^+)^{-1} (\mathcal{S}_{q,k} - S_k^+)$. It is possible to express
$\mathcal{S}_{q,k}^{-1}$ in terms of the difference of an interior and
an exterior Dirichlet-to-Neumann map, which are both bounded.

It is immediate that
\begin{equation*}
   \| \mathcal{S}_{q,k} - \mathcal{S}_{q^{\dagger},k} \| \le
   \| \mathcal{S}_{q,k} \| \,
              \| \mathcal{S}_{q^{\dagger},k} \| \,
   \| \Lambda_{q-k^2} - \Lambda_{q^{\dagger}-k^2} \| ,
\end{equation*}
while, using the statement above, it also follows that
\begin{equation}\label{e:equivstab}
   \| \Lambda_{q-k^2} - \Lambda_{q^{\dagger}-k^2} \| \le
   \| \mathcal{S}_{q,k}^{-1} \| \,
              \| \mathcal{S}_{q^{\dagger},k}^{-1} \| \,
        \| \mathcal{S}_{q,k} - \mathcal{S}_{q^{\dagger},k} \| .
\end{equation}
As a consequence, (conditional) Lipschitz stability for the
Dirichlet-to-Neumann map implies (conditional) Lipschitz stability for
the single-layer potential operator.

\section{Resolvent Estimates}\label{s:res}

Let $n \ge 3$, and let $c \in L^\infty(\mathbb R^3)$ be bounded below by a positive constant and be constant outside of a compact set. Then $-c^2\Delta$ is self-adjoint and nonnegative on $L^2(\mathbb R^n)$ with domain $H^2(\mathbb R^n)$, with respect to the inner product $\langle u, v \rangle = \int u \bar v c^{-2}$. Define the resolvent
\begin{equation}\label{e:resdef}
R_c(\lambda) := (-c^2\Delta - \lambda^2)^{-1} \colon L^2(\mathbb R^n) \to L^2(\mathbb R^n), \qquad \im \lambda>0.
\end{equation}

\begin{prop}\label{p:exp2} Let $c \in L^\infty$ be bounded below by a positive constant, and suppose $\partial_r c$ is a compactly supported measure which is bounded above by a radial measure, where $\partial_r$ is the radial vector field. There is a compact set $K \subset \mathbb R^n$ such that, for any $\chi_0, \ \chi_1 \in C_c^\infty(\mathbb R^n)$ with $\supp \chi_1 \cap K = \varnothing$, there are $C$ and $\lambda_1>0$ such that 
\begin{equation}\label{e:expres}
\|\chi_0 R_c(\lambda) \chi_0 \|_{L^2(\mathbb R^n) \to L^2(\mathbb R^n)} \le Ce^{C \re \lambda},
\end{equation}
and 
\begin{equation}\label{e:linres}
\|\chi_1 R_c(\lambda) \chi_1 \|_{L^2(\mathbb R^n) \to L^2(\mathbb R^n)} \le C/ \re \lambda,
\end{equation}
for all $\lambda$ with $\im \lambda >0$ and $\re \lambda \ge \lambda_1$.
\end{prop}

\begin{proof}
Let $h = 1/\re \lambda$. Then $R_c(\lambda) = h^2(-h^2\Delta - c^{-2})^{-1}c^{-2}$ and it is enough to show that
\begin{equation}\label{e:hest}
\|\chi_j(-h^2\Delta - c^{-2} - i \varepsilon)^{-1}\chi_j\|_{L^2(\mathbb R) \to L^2(\mathbb R)} \le 
\begin{cases} 
e^{C/h}, \qquad &j=0, \\ C/h, \qquad &j=1,
\end{cases}
\end{equation}
for all $h>0$ sufficiently small and for all $\varepsilon >0$. 

To simplify notation, in the remainder of the proof we identify radial functions on $\mathbb R^n$ with functions on $[0,\infty)$. Fix $E>0$ such that $V := E - c^{-2}$ is compactly supported. Arguing as in \cite[\S2]{d}, it suffices to construct $\varphi = \varphi_h \colon [0,\infty) \to [0,\infty)$ such that  $\varphi'$ is nonnegative with $\max \varphi'$ and $\supp \varphi'$ uniformly bounded in $h$, $\varphi'''$ is a measure, and such that
\begin{equation}\label{e:lweights}
-Ew'/2 \le  \partial_r(w(\varphi'^2 - h\varphi'' - V )), \qquad  w: = 1 - (1+r)^{-\delta},
\end{equation}
for some $\delta>0$ sufficiently small (and independent of $h$).
Indeed, once we have established \eqref{e:lweights}, we may follow \cite[\S2]{d} word by word, except that we replace \cite[(2.1)]{d} with \eqref{e:lweights}.

We will first construct $\psi = \psi_h(r) \colon [0,\infty) \to [0,\infty)$ with $\psi'$ a measure such that
\begin{equation}\label{e:mweights}
-Ew'/2 \le  \partial_r(w(\psi-V)).
\end{equation}
Fix $R>0$ such that $\supp V$ is contained in the open ball centered at zero of radius $R$. Let $\mu$ be a nonnegative, compactly supported, radial measure with $\partial_r V \le \mu$, and let 
\[
\psi = \psi_h(r) := \begin{cases}
\mu((0,r)) + \max V, \qquad &r \le R, \\
\frac B w - \frac E 2, \qquad &R<r\le R_0, \\
0, \qquad &r > R_0,
\end{cases}
\]
where $B:= w(R)(\psi(R) + E/2)$ and $R_0:=w^{-1}(2B/E)$ are taken so as to make $\psi$ continuous at $r=R, \ R_0$. Note that for the latter definition to make sense we must have $2B/E<1$ since $w$ takes values in $(0,1)$, but since $w(R) \to 0$ as $\delta \to 0^+$, we have $B \to 0$ then as well so it suffices to take $\delta>0$ sufficiently small. Then $0 \le \partial_r(w(\psi-V))$ for $r \in(0,R)\cup(R_0,\infty)$, and $-Ew'/2= (w\psi)'$ for $r \in (R,R_0)$, giving \eqref{e:mweights}.

It now remains to construct $\varphi$ as above with 
\[\varphi'^2 - h \varphi'' = \psi.\]
For this, we consider the solution to the initial value problem
\begin{equation}\label{e:ueq}
u' = (u^2-\psi)/h, \qquad u(R_0)=0.
\end{equation}
A solution exists and is absolutely continuous in a neighborhood of $R_0$ by Carath\'eodory's theorem (see e.g. \cite[Chapter 2, Theorem 1.1]{cl}), and it is unique because if $u_1$ and $u_2$ are two such solutions then the difference $\tilde u = u_1 - u_2$ solves $\tilde u' = (u_1+u_2)\tilde u$, $\tilde u(R_0) = 0$, and hence vanishes identically. 

Observe that since $\psi(r) = 0$ for all $r \ge R_0$, it follows that $u(r) = 0$ there. 
We will prove that  $0 \le u \le \sqrt{\psi(R)}$ wherever $u$ is defined. It then follows (see e.g. \cite[Chapter 2, Theorem 1.3]{cl}) that $u$ can be extended to $[0,\infty)$, where it obeys the same bounds, and we may put $\varphi':=u$.
It remains to show that  $0 \le u(r) \le \sqrt{\psi(R)}$ for $r<R_0$.

That  $u(r) \ge 0$ for $r < R_0$ follows from $u' \le u^2/h$. Indeed if there existed $r_0<R_0$ with $u(r_0)<0$ then nearby we would have $u'/u^2\le 1/h$ and hence  
\begin{equation}\label{e:upos}
u(r_0)^{-1} - u(r)^{-1}  \le (r-r_0)/h. 
\end{equation}
As $r$ increases from $r_0$ this must remain true until $u(r)$ vanishes, but as $r$ approaches the first point where $u(r)$ vanishes (and such a point must exist since $u(R_0) =0$), the left hand side of \eqref{e:upos} increases without bound, which is a contradiction.

That  $u \le \sqrt{\psi(R)}$ for $r < R_0$ follows from $u' \ge (u^2 - \psi(R))/h$ by a similar argument. Indeed, let  $v$ be the solution to
\[
v' = (v^2-\psi(R))/h, \qquad v(R_0) = 0,
\]
and observe that $v$ is defined on $\mathbb R$ and obeys $0 < v(r) < \sqrt{\psi(R)}$ for $r<R_0$. Suppose there existed $r_0<R_0$ with $u(r_0)>v(r_0)$. Let $z = u-v$, so that 
\begin{equation}\label{e:zeq}z' \ge (u^2-v^2)/h = z(u+v)/h.\end{equation}
Since $z(r_0)>0$ and $z(R_0)=0$ there must be a point $r' \in (r_0,R_0)$ such that $z(r')>0$ and $z'(r')<0$ by the mean value theorem, but this contradicts \eqref{e:zeq}, proving $u \le v < \sqrt{{\psi(R)}}$.
\end{proof}

In this paper we only use the bound \eqref{e:expres}. It would be interesting to see if the improvement \eqref{e:linres} can be used to get better estimates at high frequencies below, possibly improving part (2) of the Theorem.

\section{Forward Operator}\label{s:forward}

Beginning in this section we take dimension $n=3$. The results of this section generalize almost without changes to the case of arbitrary odd dimension $\ge 3$ and to wavespeeds $c$ which are any constant $c_0>0$ outside of a compact set -- only the notation is a little more complicated then. We expect even dimensions to also be manageable, once the behavior of the resolvent near $0$ is analyzed, e.g. in the manner of \cite{burq}.

Let $\Omega \Subset \mathbb R^3$ be a bounded domain with smooth boundary. For every $\varepsilon>0$, let $\Omega_\varepsilon$ be the set of points in $\Omega$ of distance greater than $\varepsilon$ to $\partial \Omega$, and $L^\infty_{\Omega,\varepsilon}$ be the set of functions $c \in L^\infty(\mathbb R^3)$ which are bounded below by a positive constant and which are identically 1 outside of $\Omega_\varepsilon$. For $c \in L^\infty_{\Omega,\varepsilon}$, let $R_c(\lambda)$ be the resolvent as defined in \eqref{e:resdef}. In the following Lemma we review some resolvent bounds which are essentially well-known.

\begin{lem}\label{l:res}
Let $\varepsilon>0$,  let $c \in L^\infty_{\Omega,\varepsilon}$, and fix $\chi_0 \in C_c^\infty(\mathbb R^3)$ which is identically $1$ near $\overline \Omega$.

\begin{enumerate}
\item The cutoff resolvent $\chi_0 R_c(\lambda)\chi_0 \colon L^2(\mathbb R^3) \to H^2(\mathbb R^3)$ extends continuously from $\{\lambda \in \mathbb C \mid \im \lambda>0\}$ to $\mathbb R$. 

\item For $\lambda_0\ge 1$, put
\[
a_c(\lambda_0) :=1+ \max_{\lambda \in [-\lambda_0,\lambda_0]} \|\chi_0 R_c(\lambda)\chi_0\|_{L^2(\mathbb R^3) \to H^2(\mathbb R^3)}.
\]
If $c' \in L^\infty_{\varepsilon,\Omega}$ obeys $\|c'^2-c^2\|_{L^\infty}\le 1/2a_c(\lambda_0)$, then
\[
a_{c'}(\lambda_0) \le (1 + C \|c'^2-c^2\|_{L^\infty})a_c(\lambda_0).
\]
\item For every $\chi \in C_c^\infty(\mathbb R^3)$ with $\supp \chi \cap \supp(1-c) = \varnothing$ and $\chi \chi_0 = \chi$,  $\chi_0R_c(\lambda)\chi$ extends to a bounded family of operators $H^{s}(\mathbb R^3) \to H^{s+2}(\mathbb R^3)$ for every $s \in [-2,0]$ and $\lambda \in \mathbb R$, and
\[
\max_{\lambda \in [-\lambda_0,\lambda_0]} \|\chi_0R_c(\lambda)\chi\|_{H^{s}(\mathbb R^3) \to H^{s+2}(\mathbb R^3)}\le C \lambda_0^2 a_c(\lambda_0).
\]

\item For every $\chi_1 \in C_c^\infty(\mathbb R^3)$ with $\supp \chi_1 \cap \supp \chi = \varnothing$, $\chi R_c(\lambda)\chi_1$ extends to a bounded family of operators $L^2(\mathbb R^3) \to H^{N}(\mathbb R^3)$ for every $N \in 2\mathbb N$, and 
\[
\max_{\lambda \in [-\lambda_0,\lambda_0]} \|\chi R_c(\lambda)\chi_1\|_{L^2(\mathbb R^3) \to H^{N}(\mathbb R^3)}\le C \lambda_0^{N-2} a_c(\lambda_0).
\]

\end{enumerate}
\end{lem}
\begin{proof}
To prove (1), we observe that $-c^2\Delta$ is a black box operator in the sense of Sj\"ostrand and Zworski \cite{szcomplex}, so $R_c(\lambda)\colon L^2_{\textrm{comp}}(\mathbb R^3) \to L^2_{\textrm{loc}}(\mathbb R^3)$
continues meromorphically to $\lambda \in\mathbb C$ 
(\cite[Theorem 2.2]{sres} or \cite[Theorem 4.4]{dz}).  We may replace $L^2_\textrm{loc}(\mathbb R^3)$ by $H^2_\textrm{loc}(\mathbb R^3)$  thanks to the identity $\Delta R_c(\lambda) = -c^{-2}(\lambda^2 R_c(\lambda) + I)$, so to prove (1) it remains to show that there are no poles in  $ \mathbb R$. Indeed, suppose by way of contradiction $\lambda' \in \mathbb R$ is such a pole. Then, by \cite[\S 2.4]{sres} or \cite[\S4.2]{dz} there is a corresponding outgoing resonant state, that is an outgoing solution $u_0$ to $(-\Delta - c^{-2}\lambda'^2)u_0 = 0$ which is not identically zero.
If $\lambda' = 0$, then $u_0$ is a bounded harmonic function and must vanish. If, $\lambda' \ne 0$, then by \cite[Theorem 2.4]{sres} or \cite[Theorem 3.32]{dz} $u_0$ is compactly supported and hence must vanish by Aronszajn's unique continuation theorem \cite{a}. 

To prove (2) we observe that, multiplying  by $\chi_0$ on the right and solving for $R_{c'}(\lambda)\chi_0$ in the resolvent identity
 \begin{equation}\label{e:resid}
R_{c'}(\lambda) - R_c(\lambda) = R_{c'}(\lambda)(c'^2-c^2)\Delta R_c(\lambda)
\end{equation}
gives, using the fact that $\chi_0 = 1$ near $\supp(c'-c)$,
\begin{equation}\label{e:neu}
R_{c'}(\lambda)\chi_0  =  R_c(\lambda)\chi_0\sum_{k=0}^\infty ((c'^2-c^2)\Delta \chi_0 R_c(\lambda)\chi_0)^k,
\end{equation}
where the sum is a Neumann series in the sense of operators $L^2(\mathbb R^3) \to L^2(\mathbb R^3)$. Hence
\[
 \|\chi_0 R_{c'}(\lambda)\chi_0\|_{L^2(\mathbb R^3) \to H^2(\mathbb R^3)} \le (1 + C\|c'^2-c^2\|_{L^\infty})  \|\chi_0 R_{c}(\lambda)\chi_0\|_{L^2(\mathbb R^3) \to H^2(\mathbb R^3)},
\]
and (2) follows.

To prove (3), we use the resolvent identity 
\[
R_c(\lambda) = R_1(\lambda) + R_c(\lambda)(c^2-1)\Delta R_1(\lambda) = R_1(\lambda) + R_c(\lambda)(1-c^2) + \lambda^2R_c(\lambda)(1-c^2)R_1(\lambda),
\]
where $R_1(\lambda) = (-\Delta - \lambda^2)^{-1}$. Since $(1-c)\chi = 0$, this implies
\begin{equation}\label{e:rchi}
R_c(\lambda) \chi =  R_1(\lambda) \chi + \lambda^2 R_c(\lambda)(1-c^2)R_1(\lambda)\chi.
\end{equation}
Then (3) follows from the fact that $\chi_0 R_1(\lambda) \chi_0$ is a continuous family of operators $H^{s}(\mathbb R^3) \to H^{s+2}(\mathbb R^3)$ for every $\chi_0$.

Finally, (4) has already been established for $N = 2$. It  follows for larger $N$ by induction, since
\[
\Delta \chi R_c(\lambda) \chi_1 = - \lambda^2  \chi R_c(\lambda) \chi_1  + [\Delta,\chi]R_c(\lambda)\chi_1.
\]
\end{proof}

 Denote a Fourier transform in time by
\[
\hat f (\lambda) = \int e^{-it\lambda}f(t)dt.
\]
For $\lambda_0\ge1$, let  $L_{\lambda_0}$ be the set of functions $f = f(t,x)  \in  L^2(\mathbb R; H^{-1/2}(\partial \Omega))$,  such that $\hat f(\lambda,x) \equiv 0$  when $|\lambda| \ge \lambda_0$. Let $\tau_{\partial \Omega}$ denote the trace map to $\partial \Omega$, i.e. the map which restricts a function on $\mathbb R^3$ to $\partial \Omega$. Recall that $\tau_{\partial \Omega}$ is bounded from $H^{s+1/2}(\mathbb R^3)$ to $H^{s}(\partial \Omega)$ for $s > 0$, and hence by duality $f \in L^2(\mathbb R; H^{-1/2}(\partial \Omega)) \Rightarrow f dS_{\partial \Omega} \in L^2(\mathbb R; H^{-1}(\mathbb R^3))$.

To simplify formulas below we write $f$ for $f dS_{\partial \Omega}$ below.

In the next Proposition, we apply the resolvent bounds of Lemma \ref{l:res} to study solutions to the wave equation \eqref{e:wave}.

\begin{prop}\label{p:res}
 Fix  $ \varepsilon>0$. 
 \begin{enumerate}
\item For every $c \in L^\infty_{\Omega,\varepsilon}$, the formula
 \begin{equation}\label{e:udef}
 U_cf(t) := \frac 1 {2\pi} \int e^{-it\lambda}R_c(\lambda)\hat f(\lambda)dS_{\partial \Omega}d\lambda,
 \end{equation}
 defines  a bounded linear operator $U_c \colon L_{\lambda_0} \to H^m(\mathbb R; H^{1}_{\textrm{loc}}(\mathbb R^3))$ for every $\lambda_0>0$, $m \in \mathbb N$. For every $f \in L_{\lambda_0}$, and we have
\[
(\partial_t^2  -c^2\Delta)U_cf = f .
\]
There is a constant $C$, depending on $m$, such that
\begin{equation}\label{e:restr}
\|\tau_{\partial \Omega} U_cf\|_{H^m(\mathbb R; H^{1/2}(\partial \Omega))} \le  C \lambda_0^{m+2} a_c(\lambda_0) \|f\|_{L^2(\mathbb R; H^{-1/2}(\partial \Omega))},
\end{equation}
for all $f \in L_{\lambda_0}$. 

\item If $c, \ c' \in L^\infty_{\Omega,\varepsilon}$, then for every $m \in \mathbb N, \ N \in 2 \mathbb N$, there is a constant $C$  such that, for every $\lambda_0>0$ and $f \in L_{\lambda_0}$ we have
\begin{equation}\label{e:cc'}
\|\tau_{\partial \Omega} (U_{c'} -  U_{c})f\|_{H^m(\mathbb R; H^{N-1/2}(\partial \Omega))}\le C \lambda_0^{m+N+2}a_{c}(\lambda_0)a_{c'}(\lambda_0)\|f\|_{L^2(\mathbb R; H^{-1/2}(\partial \Omega))}.
\end{equation}

\item If $c \in  L^\infty_{\Omega,\varepsilon}$, then for every $m \in \mathbb N, \ N \in 2\mathbb N$, there is  a constant $C$ such that, for every $\lambda_0 >0$ and $f \in L_{\lambda_0}$  and $c' \in L^\infty_{\Omega,\varepsilon}$ such that $\|c'^2-c^2\|_{L^\infty} \le 1/2{a_c(\lambda_0)}$ we have
\begin{equation}\label{e:frechet}
U_{c'}f(t) - U_{c}f(t) = G_{c,c'}f(t) + E_{c,c'}f(t),
\end{equation}
where
\[
G_{c,c'}f(t):= \frac 1 {2\pi} \int e^{-it \lambda}R_c(\lambda)(1 - c'^2c^{-2})\lambda^2 R_c(\lambda)\hat f(\lambda) d\lambda,
\]
and
\begin{equation}\label{e:frechest}\begin{split}
\|\tau_{\partial \Omega} G_{c,c'}f\|_{H^m(\mathbb R; H^{N-1/2}(\partial \Omega))} &\le C \lambda_0^{m+N+2}a_{c}(\lambda_0)^2 \|c^2-c'^2\|_{L^\infty(\mathbb R^3)}\|f\|_{L^2(\mathbb R; H^{-1/2}(\partial \Omega))},\\
\|\tau_{\partial \Omega} E_{c,c'}f\|_{H^m(\mathbb R; H^{N-1/2}(\partial \Omega))} &\le C \lambda_0^{m+N+2}a_{c}(\lambda_0)^3 \|c^2-c'^2\|^2_{L^\infty(\mathbb R^3)}\|f\|_{L^2(\mathbb R; H^{-1/2}(\partial \Omega))}.
\end{split}\end{equation}
So in particular 
\begin{equation}\label{e:hs}
\|\tau_{\partial \Omega} (U_{c'} -  U_{c})f\|_{H^m(\mathbb R; H^{N-1/2}(\partial \Omega))}\le  C \lambda_0^{m+N+2}a_{c}(\lambda_0)^2 \|c^2-c'^2\|_{L^\infty(\mathbb R^3)} \|f\|_{L^2(\mathbb R; H^{-1/2}(\partial \Omega))}.
\end{equation}
\end{enumerate}
\end{prop}
\begin{proof}
Fix $\chi_0, \ \chi, \ \chi_1 \in C_c^\infty(\mathbb R^3)$, such that $\chi$ is $1$ near $\partial \Omega$, $\supp \chi \cap \overline{\Omega_\varepsilon} = \varnothing$, and $\chi_0 = 1$ near $\overline \Omega \cup \supp \chi$, $\chi_1$ is $1$ near $\overline{\Omega_\varepsilon}$ and $\supp \chi_1 \cap \supp \chi = \varnothing$,
\begin{enumerate}
\item Since $f  \in H^{-1}(\mathbb R^3)$, and $\chi f = f $, for any $t \in \mathbb R$, \eqref{e:udef} defines a function $U_cf(t) \in H^{1}_{\textrm{loc}}(\mathbb R^3)$  by Lemma \ref{l:res} (3). Moreover for any $m \ge 0$ there is a constant $C$ such that
\[
\|  \chi_0 U_cf\|_{H^m(\mathbb R; H^{1}(\mathbb R^3))} \le C \lambda_0^{m+2} a_c(\lambda_0) \|f\|_{L^2(\mathbb R; H^{-1/2}(\partial \Omega))},
\] 
for all $f \in L_{\lambda_0}.$ Applying $\tau_{\partial \Omega}$ gives \eqref{e:restr}.

\item Multiplying the resolvent identity \eqref{e:resid}
on the right by $\chi$, and arguing as we did to obtain \eqref{e:rchi}, we have
\begin{equation}\label{e:resid2}
(R_{c'}(\lambda) - R_c(\lambda))\chi = R_{c'}(\lambda)(1 - c'^2c^{-2})\lambda^2 R_c(\lambda)\chi,
\end{equation}
and hence
\[
 (U_{c'} - U_c)f(t) = \frac 1 {2\pi} \int e^{-i(t-s)\lambda}  R_{c'}(\lambda)(1 - c'^2c^{-2})\lambda^2 R_c(\lambda)\hat f(\lambda)d\lambda,
\]
and in particular
\[
 \tau_{\partial \Omega}(U_{c'} - U_c)f(t) = \tau_{\partial \Omega} \frac 1 {2\pi} \int e^{-it\lambda}  \chi  R_{c'}(\lambda) \chi_1 (1 - c'^2c^{-2})\lambda^2 \chi_0R_c(\lambda)\chi\hat f(\lambda)d\lambda,
\]
Then Lemma \ref{l:res} (3) and (4) give \eqref{e:cc'}.

\item Starting with \eqref{e:neu} and arguing as we did to obtain \eqref{e:rchi}, we have
\[\begin{split}
(R_{c'}(\lambda) - R_c(\lambda))\chi &= R_c(\lambda)\sum_{k=1}^\infty ((c'^2-c^2)\Delta R_c(\lambda))^k  \chi \\
&= R_c(\lambda)\left(\sum_{k=0}^\infty ((c'^2-c^2)\Delta R_c(\lambda))^k \right)(1 - c'^2c^{-2})\lambda^2 R_c(\lambda)\chi,
\end{split}\]
and in particular
\[
(R_{c'}(\lambda) - R_c(\lambda)) f  = R_c(\lambda) \left(\sum_{k=0}^\infty ((c'^2-c^2)\Delta R_c(\lambda))^k \right)(1 - c'^2c^{-2})\lambda^2 R_c(\lambda) \hat f(\lambda) .
\]
This implies that in the decomposition \eqref{e:frechet} we have 
\[
E_{c,c'}f(t) := \frac 1 {2\pi} \int e^{-it\lambda}R_c(\lambda)\left(\sum_{k=1}^\infty ((c'^2-c^2)\Delta R_c(\lambda))^k \right)(1 - c'^2c^{-2})\lambda^2 R_c(\lambda) \hat f(\lambda)d\lambda.
\]
On the other hand
\[
\tau_{\partial \Omega}G_{c,c'}f(t)= \tau_{\partial \Omega} \frac 1 {2\pi} \int e^{-it\lambda}\chi R_c(\lambda)\chi_1 (1 - c'^2c^{-2})\lambda^2 \chi_0 R_c(\lambda) \chi \hat f(\lambda)d\lambda.
\]
Then Lemma \ref{l:res} (3) and (4) gives the first of \eqref{e:frechest}. The proof of the second of \eqref{e:frechest} (which, together with the first of \eqref{e:frechest}, implies \eqref{e:hs}) is very similar. 
\end{enumerate}
\end{proof}

\section{Landweber Iteration}\label{s:land}

Let $X$ be a Hilbert space\footnote{Following the methods of \cite{dqs}, one could also allow $X$ to be a Banach space satisfying the convexity and smoothness conditions of that paper.} and let $c\colon X \to L^\infty(\mathbb R^3)$ be Fr\'echet differentiable with locally Lipschitz derivative, and weakly sequentially continuous (in the sense that it sends weakly convergent sequences to weakly convergent sequences). Fix an open set $X_0 \subset X$  such that $c(X_0) \subset L^\infty_{\Omega,\varepsilon}$. Note that in the Theorem $c$ is the inclusion map from a finite dimensional subspace of $L^\infty(\mathbb R^3)$; in the statement there we identify $X_0$ and $c(X_0)$, and write simply $c \in X_0$.

 Fix $\lambda_0>0$, $r \ge 0$, and $w \in L^1(\mathbb R)$ which is nonnegative,  and continuous at $0$ with $w(0)=1$. For each $T>0$ and $t_0 \in \mathbb R$, let 
\[
w_T(t) = w( (t-t_0)/T),
\]
 and let $Y = L^2((\mathbb R, w_T dt ; HS(L_{\lambda_0} \to  H^r(\partial \Omega)))$, that is the Hilbert space of functions in the $w_T$-weighted $L^2$ space on $\mathbb R$ with values in the space of Hilbert--Schmidt operators from $L_{\lambda_0}$ to $H^r(\partial \Omega)$. In the Theorem, $w$ is the characteristic function of $[-1,1]$, $T = \pi/2\lambda_0$, and $r=1/2$.
 
  Define $F \colon X_0 \to Y$ by
\begin{equation}\label{e:fdef}
F(x)f(t) := \tau_{\partial \Omega} (U_{c(x)} - U_1)f(t), \qquad x \in X_0,\  f \in L_{\lambda_0}, \ t \in \mathbb R.
\end{equation}
To see that $F(x) \in Y$, observe that \eqref{e:cc'} implies that $f \mapsto F(x)f(t)$ is bounded from $L_{\lambda_0}$ to $H^M(\partial \Omega)$ for every $M \in \mathbb N$, and hence it is a Hilbert--Schmidt operator from $L_{\lambda_0}$ to $H^r(\partial \Omega)$ since the inclusion operator $H^M(\partial \Omega) \hookrightarrow H^r(\partial \Omega)$ is Hilbert--Schmidt for $M-r> 3/2= \dim\partial \Omega/2$ (see e.g. the proof of \cite[Proposition B.20]{dz}). More precisely, fix $N \in 2 \mathbb N$ with $N>r+2$ and put
\[
F_t(x) f = F(x)f(t),
\]
so that
\begin{equation}\label{e:fths}\begin{split}
\|F(x)\|^2_Y &=  \int \|F_t(x)\|^2_{HS( L_{\lambda_0} \to H^r(\partial \Omega))} w_T dt \le C  \int \|F_t(x)\|_{L_{\lambda_0} \to H^{N-1/2}(\partial \Omega)}^2 w_T dt 
\\ & \le C T \sup  \{\|\tau_{\partial \Omega} (U_{c(x)} - U_1)f\|^2_{H ^1(\mathbb R; H^{N-1/2}(\partial \Omega))} \mid f \in L_{\lambda_0}, \|f\|_{L^2(\mathbb R; H^{-1/2}(\partial \Omega))} = 1\}
\\  &\le C T \lambda_0^{2N+6}a_{c(x)}(\lambda_0)^4,
\end{split}\end{equation}
where in the second inequality we used  $\|\cdot\|_{L^\infty(\mathbb R)} \le \|\cdot\|_{H^1(\mathbb R)}$, and in the third  we used \eqref{e:cc'}.

\subsection{Fr\'{e}chet derivative}

By Proposition \ref{p:res}(3), the Fr\'echet derivative of $F$ is given by
\[
DF(x)(\tilde x- x)f =  \frac 1 \pi \tau_{\partial \Omega}  \int e^{-it \lambda}R_{c(x)}(\lambda)[Dc(x)(\tilde x- x)]c(x)^{-1}\lambda^2 R_{c(x)}(\lambda)\hat f(\lambda) d\lambda,
\]
where $Dc$ is the Fr\'echet derivative of $c$.

\begin{lem}\label{l:land} Fix $x_0 \in X_0$. There are constants $C_{\hat L}$ and $C_L$ and a small closed ball  $\mathcal B \subset X_0$ with $x_0 \in \mathcal B$, such that the following hold for all $x, \ \tilde x \in \mathcal B$, and $\lambda_0>0$:
\begin{equation}\label{e:dfbound}
\|DF(x)\|_{X \to Y} \le \hat L :=  C_{\hat L} T \lambda_0^{N+3} a_{c(x)}(\lambda_0)^2. 
\end{equation}
\begin{equation}\label{e:dflip}
\|DF(x) - DF(\tilde x)\|_{X \to Y} \le L\|x - \tilde x\|_X, \quad 
 L:=C_L T \lambda_0^{N+3}a_{c(x)}(\lambda_0)^3 .
\end{equation}
Also, $F$ is weakly sequentially closed, in the sense that if $x_n \to x$ weakly and $F(x_n) \to y$ in $Y$, then $F(x)=y$.
\end{lem}

\begin{proof}
Arguing as in \eqref{e:fths}, but using \eqref{e:frechest} in place of \eqref{e:cc'}, we have
\[\begin{split}
\|D&F(x)(x-\tilde x)\|^2_Y \\&\le C T \sup_f \left\{ \left\|  \tau_{\partial \Omega}  \int e^{-i\cdot \lambda}R_{c(x)}(\lambda)[Dc(x)(\tilde x- x)]c(x)^{-1}\lambda^2 R_{c(x)}(\lambda)\hat f(\lambda) d\lambda\right \|^2_{H^1(\mathbb R; H^{N-1/2}(\partial \Omega))}\right\}
\\& \le C T \lambda_0^{2N+6}a_{c(x)}(\lambda_0)^4 \|Dc(x)(\tilde x- x)\|_{L^\infty}^2 \le
 C T \lambda_0^{2N+6}a_{c(x)}(\lambda_0)^4 \|x -\tilde x\|_{X}^2,
\end{split}\]
where the $\sup$ is again taken over $f \in L_{\lambda_0}$ with $\|f\|_{L^2(\mathbb R; H^{-1/2}(\partial \Omega))} = 1$. This implies \eqref{e:dfbound}.

Let $h \in X$ have $\|h\|_X=1$ and $f \in L_{\lambda_0}$ have $\|f\|_{L^2(\mathbb R; H^{-1/2}(\partial \Omega))} = 1$. To save space, write $c$ for $c(x)$ and $\tilde c$ for $c(\tilde x)$. Then
\[
(DF(x) - DF(\tilde x))(h)f = \frac 1 \pi \tau_{\partial \Omega}  \int e^{-it \lambda}\left[R_{c}(\lambda)[(D c) h]c^{-1}R_{c}(\lambda) - R_{\tilde c}(\lambda)[(D \tilde c) h]{\tilde c}^{-1}R_{\tilde c}(\lambda)\right]\lambda^2\hat f(\lambda) d\lambda.
\]
Adding and subtracting $ \frac 1 \pi \tau_{\partial \Omega}  \int e^{-it \lambda} R_{c}(\lambda)[(D \tilde c) h]c^{-1}R_{c}(\lambda)\lambda^2\hat f(\lambda) d\lambda$, we estimate this in pieces. First, using $\|(Dc - D\tilde c)h\|_{L^\infty} \le C \|x-\tilde x\|_X$ (since $c$ is Fr\'echet differentiable with locally Lipschitz derivative), we have
\[
\left\| \tau_{\partial \Omega}  \int e^{-it \lambda}R_{c}(\lambda)[(D c - D{\tilde c}) h]c^{-1}R_{c}(\lambda) \lambda^2\hat f(\lambda) d\lambda\right\|_{H^1(\mathbb R; H^{N-1/2}(\partial \Omega))} \le  C \lambda_0^{N+3}a_{c(x)}(\lambda_0)^2\|x-\tilde x\|_X.
\]
Second
\[
R_{c}(\lambda)[(D \tilde c) h]c^{-1}R_{c}(\lambda) - R_{\tilde c}(\lambda)[(D \tilde c) h]{\tilde c}^{-1}R_{\tilde c}(\lambda)
\]
can be written as a sum of three differences similarly (using \eqref{e:resid} and \eqref{e:resid2}), giving
\[
\|(DF(x) - DF(\tilde x))(h)f\|_Y \le C \lambda_0^{N+3}a_{c(x)}(\lambda_0)^3\|x-\tilde x\|_X.
\]
To prove that $F$ is weakly sequentially closed, it is enough to show that if $x_n  \to x$ weakly, then $F(x_n)f$ tends to $F(x)f$ in the sense of distributions on $\mathbb R \times \partial \Omega$. As in the proof of Proposition \ref{p:res}(2), we write
\[
F(x_n)f - F(x)f = \tau_{\partial \Omega} \int e^{-it\lambda}  R_{c(x_n)}(\lambda) (c(x_n)^2 - c(x)^2) g(\lambda)d\lambda, \quad 
g(\lambda) := - \frac {\lambda^2 R_{c(x)}(\lambda) \hat f(\lambda)} {2\pi c(x)^2},
\]
and observe that $\|g(\lambda)\|_{L^\infty(\Omega_\varepsilon)}$ is uniformly bounded for $\lambda \in [-\lambda_0,\lambda_0]$. Pairing with $\varphi \in C_c^\infty(\mathbb R \times \partial \Omega)$ gives
\[
\langle F(x_n)f - F(x)f, \varphi \rangle = \int_{-\lambda_0}^{\lambda_0}  \int_{\Omega_{\varepsilon}} (c(x_n)^2 - c(x)^2) g(\lambda) \psi_n(\lambda) d\tilde xd\lambda, 
\quad 
\psi_n(\lambda):=R_{c(x_n)}(\lambda)\hat \varphi (\lambda),
\]
where $\tilde x$ is a coordinate in $\mathbb R^3$
By Lemma \ref{l:res}(4), for each $M \in \mathbb N$ the norm $\|\psi_n(\lambda)\|_{H^M(\Omega_\varepsilon)}$ is uniformly bounded for $\lambda \in [-\lambda_0,\lambda_0]$ and $n \in \mathbb N$. Passing to a subsequence, $\psi_n(\lambda)$ converges uniformly in $H^M(\Omega_\varepsilon)$ to a limit, which we denote $\psi(\lambda)$. Then, as $n \to \infty$,
\[
\left| \int_{-\lambda_0}^{\lambda_0} \int_{\Omega_{\varepsilon}} (c(x_n)^2 - c(x)^2) g(\lambda) (\psi_n(\lambda) - \psi(\lambda))d \tilde xd\lambda\right| \le C(\lambda_0)\sup_{\lambda} \|\psi_n(\lambda) - \psi\|_{L^\infty} \to 0,
\]
and
\[
 \int_{-\lambda_0}^{\lambda_0} \int_{\Omega_{\varepsilon}} (c(x_n)^2 - c(x)^2) g(\lambda)  \psi(\lambda)d \tilde xd\lambda \to 0
\]
by the dominated convergence theorem, since $\int_{\Omega_{\varepsilon}} (c(x_n)^2 - c(x)^2) g(\lambda)  \psi(\lambda)d \tilde x$ is uniformly bounded in $\lambda$ and $n$ and tends to $0$ for each $\lambda$ since $c(x_n) \to c(x)$ weakly (this follows from $x_n \to x$ weakly since $c$ is weakly sequentially continuous).
This shows that every subsequence of $\langle F(x_n)f - F(x)f, \varphi \rangle$ has a subsequence which tends to $0$, proving that the original sequence tends to $0$.
\end{proof}

\subsection{Hilbert-Schmidt misfit functional and gradient}

We use the misfit functional
\[
\|F(x) - F(x^\dagger)\|_Y^2 = \sum_{j=1}^\infty \int \|\tau_{\partial \Omega} (U_{c(x)} - U_{c(x^\dagger)})\psi_j\|^2_{H^r(\partial \Omega)} w_T dt,
\]
where $x^\dagger \in \mathcal B$ is the ``true'' model, and where $\{\psi_j\}_{j=1}^\infty$ is any orthonormal basis of $L_{\lambda_0}$.

Writing $c$ for $c(x)$ , the gradient $DF(x)^*$ is given by
\[\begin{split}
\langle DF(x)^*&(y-\tilde y),(x-\tilde x)  \rangle_X = \langle (y-\tilde y),DF(x)(x-\tilde x)  \rangle_Y \\
&= \frac 1 \pi \sum_{j=1}^\infty \int \left\langle (y-\tilde y)\psi_j, \tau_{\partial \Omega}  \int e^{-it \lambda}R_{c}(\lambda)[(Dc)(x- \tilde x)]c^{-1}\lambda^2 R_{c}(\lambda)\widehat{\psi_j}(\lambda) d\lambda  \right\rangle_{H^r(\partial \Omega)} w_T dt\\
&= \sum_{j=1}^\infty \int \!\!\!\int \frac {(Dc)(x - \tilde x)} {\pi c}  \lambda^2 R_{c}(-\lambda)[(1 - \Delta_{\partial \Omega})^{r}w_T (y - \tilde y)\psi_j]\widehat{\phantom{x}}(\lambda) \left[R_{c}(-\lambda) \overline{\widehat{\psi_j}}(\lambda) \right] d\lambda d \tilde x,
\end{split}\]
where $\tilde x$ is a coordinate on $\mathbb R^3$ and
 $\Delta_{\partial \Omega}$ is the nonpositive Laplacian on $\partial \Omega$.
Writing $c_m$ for $c(x_m)$, we have
\[\begin{split}
D&F(x_m)^*(F(x_m)-F(x^\dagger)) =  \\
&(Dc_m)^* \sum_{j=1}^\infty  \frac 1{\pi c_m} \int  \lambda^2 \left[R_{c_m}(-\lambda)[(1 - \Delta_{\partial \Omega})^{r}w_T (F(x_m)-F(x^\dagger))\psi_j]\widehat{\phantom{x}}(\lambda) \right]\left[R_{c_m}(-\lambda) \overline{\widehat{\psi_j}}(\lambda) \right]d\lambda. 
\end{split}\]
Note that in the Theorem $c$ is an inclusion map, so $(D c_m)^*$ is a projection (onto piecewise constant wavespeeds).

\subsubsection*{Adjoint state equation}

The first factor in the cross correlation integral above can be written as the solution to an inhomogeneous wave equation solved backwards in time, namely
\[
u(t) := \frac 1 {2\pi} \int e^{-it\lambda}R_{c(x_m)}(-\lambda)[(1 - \Delta_{\partial \Omega})^{r}w_T (F(x_m)-F(x^\dagger))\psi_j]\widehat{\phantom{x}}(\lambda) d\lambda
\]
is the backwards in time solution to
\[
(\partial_t^2 - c^2\Delta)u =  \left[(1 - \Delta_{\partial \Omega})^{r}w_T (F(x_m)-F(x^\dagger))\psi_j\right](-t).
\]
In \cite{bss}, Blazek, Stolk and Symes show that the
analogous equation for a problem without frequency ``bandlimitation'' does not have
a solution.

  \subsection{Landweber iteration}

 We define the Landweber iterates by the equation
\begin{equation}\label{e:landit}
x_{m+1} = x_m - \mu DF(x_m)^*(F(x_m)-F(x^\dagger)),
\end{equation}
where the step size $\mu$ is sufficiently small (see \cite[(3.5)]{dqs}.

Suppose the inversion has uniform H\"older-type stability in the sense that there is a constant $C_F$ such that for every $x, \ \tilde x \in \mathcal B$ we have
\begin{equation}\label{e:hstab}
\frac 1 {\sqrt 2} \|x - \tilde x\|_X \le C_F \|F(x) - F(\tilde x)\|_Y^{\frac{1+\epsilon}2},
\end{equation}
for some $\epsilon \in (0,1]$. Note that below we will only use the case $\epsilon=1$. Then \cite[Theorem 3.2]{dqs} applies and the Landweber iteration converges: see the next section for details.

\section{Proof of Theorem and Discussion: Convergence}\label{s:conv}

In this section we work under the assumptions of \S\ref{s:land} and the additional assumptions of the Theorem; in particular $X = \mathbb R^N$, $w$ is the chacteristic function of $[-1,1]$, $T=\pi/2\lambda_0$, and $r=1/2$. We apply the following convergence result, a consequence of \cite[Theorem~3.2]{dqs}.

\begin{prop}[{\cite[Theorem 3.2]{dqs}}]\label{p:dqs}
Let $X$ and $Y$ be Hilbert spaces, $\mathcal B$ a closed ball in $X$, and $F\colon \mathcal B \to Y$ continuous and Fr\'echet differentiable with Lipschitz derivative. Suppose further that $F$ is weakly sequentially closed and that there are constants $L,\ \hat L, \ C_F\ge1$  such that
\[
\|DF(x)\|_{X \to Y}\le \hat L, \quad \|DF(x)- DF(\tilde x)\|_{X \to Y}\le L,
\]
and \eqref{e:hstab} hold for all $x, \ \tilde x \in \mathcal B$, with $\epsilon=1$ in the case of \eqref{e:hstab}. Let $\mu \in(0, \min\{1/2 \hat L^2,4C_F^2\})$ and let $\mathcal B_1 \subset \mathcal B$ be a closed ball of radius  $R \le 1/2C_F\sqrt{L\hat L}$. Then for any $x_0, x^\dagger \in \mathcal B_1$, the sequence of Landweber iterates, defined by \eqref{e:landit}, converges to $x^\dagger$ at the following exponential rate:
\[
\|x_k - x^\dagger\|_X \le R \left(1 - \frac \mu{4C_F^2}\right)^{k/2}.
\]
\end{prop}

Now the Theorem follows from Proposition \ref{p:dqs}. All the assumptions apart from \eqref{e:hstab} follow from Lemma \ref{l:land}, and \eqref{e:hstab} is deduced below from \cite[Theorem 2.7]{BdHQ}. To prove part (2) of the Theorem, we observe that \eqref{e:expres} implies that $L,\ \hat L$ grow at most exponentially in $\lambda_0$, while the proof of \eqref{e:hstab} below shows that $C_F$ grows at most like  $\sqrt{\lambda_0}$.

\begin{proof}[Proof of \eqref{e:hstab}] 
To save space we write $c = c(x), \ \tilde c = c(\tilde x)$.
Fix $\lambda_1 \in (0,1)$ such that $\lambda_1^2$ is not a Dirichlet eigenvalue of $-\Delta$ or of $-\Delta + \lambda_1^2(1-c^{-2})$ or of $-\Delta + \lambda_1^2(1-\tilde c^{-2})$ for any $x, \tilde x \in \mathcal B$ (this can be arranged by taking the ball $\mathcal B$ small enough).

As an orthonormal basis of $L_{\lambda_0}$, take $\{\psi_j\}_{j=1}^\infty$ such that 
$
\{\hat \psi_j\}_{j=1}^\infty = \{\hat a_\ell b_m\}_{(\ell, m) \in \mathbb Z \times \mathbb N},
$
 where each $\hat a_\ell$ is the characteristic function of $[-\lambda_0,\lambda_0]$ multiplied by $e^{-i \ell \pi \lambda/\lambda_0}/\sqrt{2\lambda_0}$, and  $\{b_m\}_{m=1}^\infty$ is an orthonormal basis of $H^{-1/2}(\partial \Omega)$ with $b_1$ chosen such that
\[
\frac{\| (\mathcal{S}_{\lambda_1^2(1-c^{-2}),\lambda_1} - \mathcal{S}_{\lambda_1^2(1-\tilde c^{-2}),\lambda_1})b_1\|_{H^{1/2}(\partial\Omega)}}{\| (\mathcal{S}_{\lambda_1^2(1-c^{-2}),\lambda_1} - \mathcal{S}_{\lambda_1^2(1-\tilde c^{-2}),\lambda_1})\|_{H^{-1/2}(\partial\Omega) \to H^{1/2}(\partial\Omega)}} \ge 1/2.
\]
By \eqref{e:equivstab} we have
\[
\frac{ \| (\mathcal{S}_{\lambda_1^2(1-c^{-2}),\lambda_1} - \mathcal{S}_{\lambda_1^2(1-\tilde c^{-2}),\lambda_1})\|_{H^{-1/2}(\partial\Omega) \to H^{1/2}(\partial\Omega)}}  {\| (\Lambda_{-\lambda_1^2c^{-2}} - \Lambda_{-\lambda_1^2\tilde c^{-2}})\|_{H^{1/2}(\partial\Omega) \to H^{-1/2}(\partial\Omega)}} \ge 
1/C,
 \]
 where we used the fact that if the ball $\mathcal B$ is small enough then $\| \mathcal{S}_{\lambda_1^2(1-c^{-2}),\lambda_1}\|/2 \le   \|\mathcal{S}_{\lambda_1^2(1-\tilde c^{-2}),\lambda_1}\| \le 2\| \mathcal{S}_{\lambda_1^2(1-c^{-2}),\lambda_1}\|$ by \eqref{e:neu}.
 By  \cite[Theorem 2.7]{BdHQ}, we have
 \[
 \| (\Lambda_{-\lambda_1^2c^{-2}} - \Lambda_{-\lambda_1^2\tilde c^{-2}})\|_{H^{1/2}(\partial\Omega) \to H^{-1/2}(\partial\Omega)} \ge \|x-\tilde x\|_X/C.
 \]
  On the other hand, if $J$ is a sufficiently small open interval containing $\lambda_1$, we have
  \[
  \min_{\lambda \in J}\|\tau_{\partial \Omega} (R_{c}(\lambda) - R_{\tilde c}(\lambda)) b_1\|_{H^{1/2}(\partial\Omega)} \ge \|\tau_{\partial \Omega} (R_{c}(\lambda_1) - R_{\tilde c}(\lambda_1))b_1\|_{H^{1/2}(\partial\Omega)}/2.
  \]
 Meanwhile, since $\hat a_\ell(\lambda) = e^{-i\ell \pi \lambda/\lambda_0} \hat a_0(\lambda)$, we have $U_c a_\ell b_m (t) =U_c a_0 b_m (t+\ell\pi/\lambda_0)$, so that
 \[
  \|F(x) - F(\tilde x)\|^2_Y \ge \sum_{\ell \in \mathbb Z} \int \|\tau_{\partial \Omega} (U_{c} - U_{\tilde c})a_\ell b_1\|_{H^{1/2}(\partial\Omega)}^2 w_T dt = \int \|\tau_{\partial \Omega} (U_{c} - U_{\tilde c})a_0 b_1\|_{H^{1/2}(\partial\Omega)}^2 dt.
 \]
 Putting these estimates together gives
  \[\begin{split}
  \|F(x) - F(\tilde x)\|^2_Y &\ge  \frac 1 {8\pi^2\lambda_0} \int_{-\lambda_0}^{\lambda_0} \|\tau_{\partial \Omega} (R_{c}(\lambda) - R_{\tilde c}(\lambda))b_1\|_{H^{1/2}(\partial\Omega)}^2 d\lambda\\
  & \ge \|\tau_{\partial \Omega} (R_{c}(\lambda_1) - R_{\tilde c}(\lambda_1))b_1\|_{H^{1/2}(\partial\Omega)}^2/C\lambda_0 \\
  & =  \| (\mathcal{S}_{\lambda_1^2(1-c^{-2}),\lambda_1} - \mathcal{S}_{\lambda_1^2(1-\tilde c^{-2}),\lambda_1})b_1\|_{H^{1/2}(\partial\Omega)}^2/C \lambda_0  \\
  & \ge \|x-\tilde x\|_X^2/C \lambda_0.
  \end{split}\]
\end{proof}

\def\arXiv#1{\href{http://arxiv.org/abs/#1}{arXiv:#1}}


\begin{thebibliography}{0}

\bibitem[Al]{Alessandrini1988}   Giovanni Alessandrini, \textit{Stable determination of conductivity by boundary measurements},
   {Appl. Anal.}, 27:1--3, (1988), pp.
  {153--172}.

\bibitem[AlVe]{Alessandrini2005}  Giovanni Alessandrini and Sergio Vessella, \textit{{Lipschitz} stability for the inverse conductivity problem},
{Adv. in Appl. Math.}, 35:2, (2005), pp.
 {207--241}.



\bibitem[Ar]{a} N. Aronszajn, \textit{A unique continuation theorem for solutions of elliptic partial differential equations or inequalities of second order}, J. Math. Pures Appl. (9), 36 (1957), pp. 235--249.

\bibitem[BCL1]{Bamberger1977} A. Bamberger, G. Chavent, and P. Lailly, \textit{Une application de la th\'eorie du contr\^ole \`a un probl\`eme inverse
	de sismique}, {Annales de G\' eophysique},  33 (1977), pp. 183--200.

\bibitem[BCL2]{Bamberger1979} A. Bamberger, G. Chavent, and P. Lailly, \textit{About the stability of the inverse problem in the 1-d wave equation}, Journal of Applied Mathematics and Optimisation, 5 (1979), pp. 1--47.

\bibitem[BHL]{Bao2007}  Gang Bao, Songming Hou, and Peijun Li, \textit{Inverse scattering by a continuation method with initial guesses from a direct imaging algorithm}, J. Comput. Phys. 227:1 (2007), pp. 755--762.

\bibitem[BaLi]{Bao2005a} Gang Bao and Peijun Li, \textit{Inverse medium scattering problems for electromagnetic waves}, SIAM J. Appl. Math., 65:5 (2005), pp. 2049--2066 (electronic).

\bibitem[BaTr]{Bao2010} Gang Bao and Faouzi Triki, \textit{Error estimates for the recursive linearization of inverse medium problems}, J. Comput. Math., 28:6 (2010), pp. 725--744.


\bibitem[BCJ]{Beasley1998} C.J. Beasley, R.E. Chambers,  and  Z. Jiang,
  \textit{A new look at simultaneous sources},
   {Expanded Abstracts},
  {68th Ann. Internat. Mtg.},
   {Society of Exploration Geophysicists},
  (1998),
  pp. {133--135}



\bibitem[BOV]{Hadj-Ali2008} Hafedh Ben-Hadj-Ali, St\'ephane Operto, and Jean Virieux, \textit{Velocity model-building by 3D frequency-domain, full-waveform inversion of wide-aperture seismic data}, Geophysics, 73:5 (2008), 5285980.



\bibitem[BdHQ]{BdHQ} Elena Beretta, Maarten V. de Hoop, and Lingyun Qiu,
\textit{Lipschitz stability of an inverse boundary value problem for a Schrödinger-type equation}, SIAM J. Math. Anal. 45 (2013), no. 2, pp. 679-–699. 


\bibitem[BeFr]{Beretta2011} Elena Beretta and Elisa Francini, \textit{{Lipschitz} stability for the electrical impedance tomography problem: the complex case},
 {Comm. Partial Differential Equations}, 36:10, (2011), pp.
{1723--1749}.


\bibitem[Be]{Berkhout2008-1} A. J. Berkhout, \textit{Changing the mindset in seismic data acquisition}, The Leading Edge, 27 (2008), pp. 924--938.

\bibitem[BBV]{Berkhout2008-2} A. J. Berkhout, G. Blacquire, and E. Verschuur, \textit{From simultaneous shooting to blended acquisition}, in
Expanded Abstracts, vol. 27, Society of Exploration Geophysicists, 2008, pp. 2831--2838.

\bibitem[BSS]{bss} Kirk D Blazek, Christiaan Stolk, and William W Symes, \textit{A mathematical framework for inverse wave problems
in heterogeneous media}, Inverse Problems, 29:6 (2013), 065001, 37 pp.

\bibitem[BTT]{Bozdag2011} E. Bozdag, J. Trampert, and J. Tromp, \textit{Misfit functions for full waveform inversions based on instantaneous
	phase and envelope measurements}, Geophysical Journal International, 185 (2011), pp. 845--870.

\bibitem[Br]{Brossier2009} R. Brossier, \textit{Imagerie sismique \`a deux dimensions des milieux viso-\'elastiques
	par inversion des formes d'ondes: d\'eveloppements m\'ethodologiques
	et applications}, PhD Thesis, Universit\'e de Nice-Sophia Antipolis, 2009.

\bibitem[BSZC]{Bunks1995} C. Bunks,  F. Saleck, S. Zaleski, and G. Chavent, \textit{Multiscale seismic waveform inversion}, Geophysics, 60:5 (1995), pp. 1457--1473.


\bibitem[Bu1]{burq} Nicolas Burq, \textit{ D\'ecroissance de l'\'energie locale de l'\'equation des ondes pour le probl\`eme ext\'erieur et absence de r\'esonance au voisinage du r\'eel},  Acta Math., 180:1 (1998), pp. 1--29.
 
\bibitem[Bu2]{bu}
Nicolas Burq.
\textit{ Lower bounds for shape resonances widths of long range
  {S}chr\"odinger operators},
 Amer. J. Math., 124:4 (2002), pp. 677--735.

 
 \bibitem[CaVo]{cv} Fernando Cardoso and Georgi Vodev,
 \newblock \textit{Uniform estimates of the resolvent of the {L}aplace-{B}eltrami
   operator on infinite volume {R}iemannian manifolds {II}},
 \newblock Ann. Henri Poincar\'e 3:4, (2002), pp. 673--691.

\bibitem[CDR]{Chen2009} Yu Chen, Ran Duan,  and Vladimir Rokhlin, \textit{On the inverse scattering problem in the acoustic environment}, J. Comput. Phys. 228:9 (2009) pp. 3209--3231.

\bibitem[CoLe]{cl} Earl A. Coddington and Norman Levinson, \textit{Theory of ordinary differential equations}, McGraw--Hill, New York, 1955.

\bibitem[DaSc]{Dai2009} W. Dai and G.T. Schuster, \textit{Least-squares migration of simultaneous sources data
           with a deblurring filter}, Expanded Abstracts, 79th Ann. Internat. Mtg., Society of Exploration Geophysicists, (2009), pp. 2990--2994.
           



\bibitem[Da]{d} Kiril Datchev, \textit{Quantitative limiting absorption principle in the semiclassical limit.} Geom. Funct. Anal., 24:3 (2014), pp. 740--747.

\bibitem[dHQS]{dqs} Maarten V de Hoop, Lingyun Qiu, and Otmar Scherzer, \textit{Local analysis of inverse problems: H\"older stability and iterative reconstruction}, Inverse Problems, 28:4 (2012), 045001.

\bibitem[DyZw]{dz} Semyon Dyatlov and Maciej Zworski, \textit{Mathematical theory of scattering resonances}, \url{http://math.mit.edu/~dyatlov/res/res.pdf}.

\bibitem[GAW]{Gao2010} F. Gao, A. Atle, and P. Williamson, \textit{Full waveform inversion using deterministic source encoding} {Expanded Abstracts},
{Ann. Internat. Mtg.},
 {Society of Exploration Geophysicists} 29 (2010), pp.
 {1013--1017}.

\bibitem[HPYS]{Ha2010} W. Ha, S. Pyun, J. Yoo, and C. Shin, \textit{Acoustic full waveform inversion of synthetic land and marine data
	in the Laplace domain}, Geophysical Prospecting 58 (2010) pp. 1033--1047.



\bibitem[KAHNLBL]{Krebs2009}  J. Krebs, J. Anderson, D. Hinkley, R.
            Neelamani, S. Lee, A. Baustein, and M.-D.
            Lacasse,
  \textit{Fast full-wavefield seismic inversion using encoded
           sources},
   {Geophysics} 74:6 (2009), pp.
{WCC177--WCC188}



\bibitem[La]{Lailly1983} P. Lailly, \textit{The seismic inverse problem as a sequence of before stack migrations}, Conference on Inverse Scattering: Theory and Application,Society for Industrial and Applied Mathematics (1983) pp. 206--220.


\bibitem[LaTr]{Lasiecka1983}  Irena Lasiecka and  Roberto Triggiani, \textit{Regularity of hyperbolic equations under {$L_{2}(0,\,T;L_{2}(\Gamma ))$}-{D}irichlet boundary terms}, Appl. Math. Optim., 10:3, (1983), pp. 275--286.

\bibitem[Ma]{Mandache2001} Niculae Mandache, \textit{Exponential instability in an inverse problem for the {S}chr\"odinger equation},
  {Inverse Problems}, 17:5 (2001), pp. 1435--1444.




\bibitem[Na]{Nachman1988-1} Adrian I. Nachman, \textit{Reconstruction from boundary measurements}, Ann. of Math. (2), 128 (1988), pp. 531--576.

\bibitem[NKKRDA]{Neelamani2010}{R. Neelamani, C. Krohn, J. Krebs, J.
            Romberg, M. Deffenbaugh, and J. Anderson},
  \textit{Efficient seismic forward modelling using simultaneous
           random sources and sparsity},
  {Geophysics} 75:6 (2010), pp.
{WB15--WB27}.

\bibitem[No]{Novikov2011} Roman G. Novikov, \textit{New global stability estimates for the {G}el'fand-{C}alderon inverse problem},
   {Inverse Problems}, 27:1, (2011), {015001, 21 pp.}





\bibitem[PPO]{Pan1988} G. S. Pan, R. A. Phinney, and R. I. Odom, \textit{Full-waveform inversion of plane-wave seismograms in stratified acoustic
	media: Theory and feasibility}, Geophysics, 53 (1988), pp. 21--31.


\bibitem[PrWo]{Pratt1990} R. G. Pratt and M. H. Worthington, \textit{ Inverse theory applied to multi-source cross-hole tomography. Part
	1: Acoustic wave-equation method}, Geophysical Prospecting 38 (1990), pp. 287-310.
	
\bibitem[PrGo]{Pratt1991} G. Pratt and N. Goulty, \textit{Combining wave-equation imaging with traveltime tomography to form
	high-resolution images from crosshole data}, Geophysics 56 (1991), pp. 208--224.
	
\bibitem[RiBe]{Rietveld1994} W.E.A. Rietveld and A.J. Berkhout,
  \textit{Prestack depth migration by means of controlled illumination},
 {Geophysics}, 59:5 (1994), pp.
  {801--809}.
  
\bibitem[RoTa]{rt} Igor Rodnianski and Terence Tao, \textit{Effective limiting absorption principles, and applications}, Commun. Math. Phys., 333:1 (2015), pp. 1--95.


\bibitem[ShCh]{Shin2009} C. Shin and Y. H. Cha, \textit{Waveform inversion in the Laplace-Fourier domain}, Geophysical Journal International, 177 (2009), pp. 1067--1079.


\bibitem[SiPr]{Sirgue2004} L. Sirgue and R. Pratt, {Efficient waveform inversion and imaging: A strategy for selecting temporal frequencies}, Geophysics, 69:1 (2004), pp. 231--248.

\bibitem[Sj]{sres} Johannes Sj\"ostrand, \textit{Lectures on resonances},  
\url{http://sjostrand.perso.math.cnrs.fr/Coursgbg.pdf}.


\bibitem[SjZw] {szcomplex} Johannes Sj\"ostrand and Maciej Zworski, \textit{Complex scaling and the distribution of scattering poles}, J. Amer. Math. Soc., 4:4 (1991), pp. 729--769.

\bibitem[SHH]{Stefani2007} J. Stefani, G. Hampson, and E. F. Herkenhoff, \textit{Acquisition using simultaneous sources},
  {Extended Abstracts},
  {69th Conference and Exhibition},
  {European Association of Geoscientists and Engineers}, (2007), B006.

\bibitem[SyUh] {su} John Sylvester and Gunther Uhlmann, \textit{A global uniqueness theorem for an inverse boundary value problem}, Ann. of Math. (2) 125 (1987), no. 1, pp. 153-–169.

\bibitem[Sy]{Symes2009} W. W. Symes, \textit{The seismic reflection inverse problem}, Inverse Problems, 25:12, (2009), 123008, 39 pp.

\bibitem[Ta1]{Tarantola1984} A. Tarantola, \textit{Inversion of seismic reflection data in the acoustic approximation}, Geophysics, 49 (1984), pp. 1259--1266.

\bibitem[Ta2]{Tarantola1987} A. Tarantola, \textit{Inverse Problem Theory: methods for data fitting and and model parameter
	estimation}, Elsevier, 1987.

\bibitem[ViSt]{Vigh2008} D. Vigh and W. Starr,
  \textit{3{D} prestack plane-wave, full-waveform inversion},
  {Geophysics}, 73:5 (2008), pp. VE135--VE144.


\bibitem[WdHX] {w} S. Wang, M. V. de Hoop, and J. Xia, \textit{On 3D modeling of seismic wave propagation via a structured massively
parallel multifrontal direct Helmholtz solver}, Geophys. Prosp., 59 (2011), pp. 857--873.

\bibitem[Wh] {Whitmore1995} N. D. Whitmore, \textit{An imaging hierarchy for common angle plane wave
           seismograms}, PhD Thesis, University of Tulsa, 1995.

\bibitem[YuSi]{Simons}   Yanhua O. Yuan and Frederik J. Simons Multiscale adjoint                             
  waveform-difference tomography using wavelets Geophysics, 79, (2014), pp.              
  W79--W95. 
\end{thebibliography}
\end{document}